\documentclass[10pt]{amsart}

\usepackage[utf8]{inputenc}
\usepackage{mathtools}
\usepackage[marginpar]{todo}

\usepackage{tikz}
\usetikzlibrary{calc, matrix}

\usepackage[pdftex]{hyperref}

\newtheorem{theorem}{Theorem}[section]

\newtheorem{corollary}[theorem]{Corollary}
\newtheorem{lemma}[theorem]{Lemma}

\theoremstyle{definition}
\newtheorem{definition}[theorem]{Definition}

\newtheorem{remark}[theorem]{Remark}

\newcommand{\BC}{\mathbf{C}}
\newcommand{\BP}{\mathbf{P}}
\newcommand{\BQ}{\mathbf{Q}}
\newcommand{\BR}{\mathbf{R}}

\newcommand{\BZ}{\mathbf{Z}}

\newcommand{\CC}{\mathcal{C}}

\newcommand{\CL}{\mathcal{L}}

\newcommand{\CO}{\mathcal{O}}

\newcommand{\norm}[1]{\left| #1 \right|}
\newcommand{\normsq}[1]{\norm{#1}^2}

\newcommand{\fg}{\mathfrak{g}}

\newcommand{\fgd}{\mathfrak{g}^\ast}

\newcommand{\red}{/\!/}
\newcommand{\reda}[1]{{\red\!}_{#1}}

\newcommand{\rred}{/\!/\!/}

\newcommand{\omegar}{\omega_\BR}
\newcommand{\omegac}{\omega_\BC}
\newcommand{\mur}{\mu_\BR}

\newcommand{\muc}{\mu_\BC}      
  
\newcommand{\muhk}{\mu_\mathrm{G,s}}

\newcommand{\iso}{\cong}

\newcommand{\bi}{\mathbf{i}}
\newcommand{\bj}{\mathbf{j}}
\newcommand{\bk}{\mathbf{k}}

\newcommand{\dif}[1][]{\operatorname{d}_{#1} \!}

\begin{document}
\title[Kirwan surjectivity and Lefschetz-Sommese theorems]{Kirwan surjectivity and Lefschetz-Sommese theorems for a generalized hyperkähler reduction}


\author[J. Fisher]{Jonathan Fisher}
\email{jonathan.m.fisher@gmail.com}

\author[L. Jeffrey]{Lisa Jeffrey}
\address{Department of Mathematics, University of Toronto, Canada}
\email{jeffrey@math.toronto.edu}

\author[A. Malus\`a]{Alessandro Malus\`a}
\address{Department of Mathematics, University of Toronto, Canada}
\email{amalusa@math.utoronto.ca}

\author[S. Rayan]{Steven Rayan}
\address{Centre for Quantum Topology and Its Applications (quanTA) and Department of Mathematics and Statistics, University of Saskatchewan, Canada}
\email{rayan@math.usask.ca}

\begin{abstract}
Let $G$ be a compact Lie group.
We study a class of Hamiltonian $(G \times S^{1})$-manifolds \emph{decorated} with a function $s$ with certain equivariance properties, under conditions on the $G$-action which we call of \emph{(semi-)linear type}.
In this context, a close analogue of hyperkähler reduction is defined, and our main result establishes surjectivity of an appropriate analogue of Kirwan’s map.
As a particular case, our setting includes a class of hyperkähler manifolds with trihamiltonian torus actions, to which our surjectivity result applies.
\end{abstract}

\maketitle


\section{Introduction}

\renewcommand{\thetheorem}{\arabic{theorem}}
In light of the remarkable success of the Atiyah-Bott-Kirwan theory for symplectic quotients \cite{YangMillsRiemannSurface,Kirwan}, it seems natural to consider Morse theory on a hyperkähler manifold $M$ with function $\normsq{\muhk}$, the norm squared of the trimoment map of a group action, in order to study the cohomology ring of the corresponding hyperk\"ahler quotient.
However, one immediately bumps into trouble due to the non-compactness of $M$.
Under the assumption that the gradient flow of $\normsq{\muc}$ converges, the requisite Morse theory was developed in \cite{JKK}.
Partial results on the convergence of the gradient flow were obtained by the first author \cite{FisherMorse}, but it appears that in general this convergence is difficult to establish.

In their study of hyperpolygon spaces, which arise as hyperkähler quotients by a compact group $G$, two of the authors~\cite{FisherRayan} established Kirwan surjectivity using circle compactness with respect to an additional $S^{1}$-action.
The construction does not, in fact, use the whole hyperkähler structure, but only one Kähler form for which the $S^{1}$-action is Hamiltonian.
Correspondingly, the quotient is viewed as the Kähler reduction of the subvariety cut by the complex moment map for $G$, regarded simply as an equivariant holomorphic function without using its relation to the group action.

In this work, we aim to extend that technique to a more general setting and obtain similar results in cohomology, which is always understood to take rational coefficients unless otherwise specified.
Namely, we consider a Kähler manifold $M$, acted on by $G \times S^{1}$ with moment map $\mu = (\mu_{G}, \mu_{S^{1}})$, and equipped with a $G$-invariant holomorphic function $s$.
Assuming that $s$ is homogeneous of some positive degree with respect to the $S^{1}$-action, we call this a \emph{decorated $G$-manifold}.
Under appropriate regularity conditions we define the \emph{hyper-reduction} of $M$ to be
\begin{equation}
	\label{eq:reduction}
	M \rred G = \bigl(\mu^{-1}(0) \cap s^{-1}(0) \bigr) / G \, ,
\end{equation}
which can be regarded as a symplectic quotient of the complex analytic subvariety $s^{-1}(0)$ and is in general a non-compact complex orbifold, smooth if $G$ acts freely on the appropriate locus.
The embeddings
\begin{equation}
	\mu^{-1}(0) \cap s^{-1}(0) \hookrightarrow \mu^{-1} (0) \hookrightarrow M
\end{equation}
induce natural maps
\begin{equation}
	\begin{gathered}
		\kappa \colon H^{*}_{G}(M) \to H^{*}_{G} \bigl(\mu^{-1}(0)\bigr) \iso H^{*} (M \red G) \, , \\
		\kappa_{G,s} \colon H^{*}_{G} (M) \to H^{*}_{G} \bigl(\mu^{-1}(0) \cap s^{-1}(0) \bigr) \iso H^{*} (M \rred G)
	\end{gathered}
\end{equation}
in rational cohomology, which we call the \emph{Kirwan map} and \emph{hyper-Kirwan map}, respectively.
Since $s$ is $G$-invariant, it also descends to a function on $M \red G$, and $M \rred G$ can equivalently be described as the zero locus of the function on the quotient.
The inclusion $\iota \colon M \rred G \hookrightarrow M \red G$ gives then a factorisation of the hyper-Kirwan map as
\begin{equation}
	H^{*}_{G} (M) \stackrel{\kappa}{\longrightarrow} H^{*} (M \red G) \stackrel{\iota^{*}}{\longrightarrow} H^{*} (M \rred G) \, .
\end{equation}
We approach surjectivity of $\kappa_{G,s}$ by studying the two maps $\kappa$ and $\iota^{*}$ separately.
As it turns out, our results actually hold under weaker conditions on $M$, which we may only require to be \emph{almost} Kähler or, for some parts, just symplectic.
Our main results are as follows.

\begin{theorem}[Cf. Theorem~\ref{thm:KirwanSurj}]
	\label{thm:kirwan_surj_intro}
	Suppose $(M, \omega)$ is a Hamiltonian $G$-manifold, equipped with an additional Hamiltonian $S^{1}$-action commuting with $G$.
	If $0 \in \fgd$ is regular for $\mu_{G}$,  $M$ and $M \red G$ are circle compact (see Definition~\ref{def:cc}), and $\overline{M \red G}$ is smooth, then the Kirwan map $\kappa \colon H^{*}_{G} (M) \to H^{*} (M \red G)$ is surjective.
\end{theorem}

\begin{theorem}[Cf. Theorem~\ref{thm:equivariant_sommese_quotient}]
	\label{thm:eq_sommese_intro}
	Suppose $M$ is an almost Kähler manifold equipped with a Hamiltonian $(G \times S^{1})$-action and a holomorphic function $s$ valued in some complex vector space $V$ such that
	\begin{enumerate}
		\item $s$ is $G$-invariant and homogeneous of some positive degree $d$ for the $S^{1}$-action;
		\item $0$ is a regular value of both $s$ and $\mu_{G}$, $\mu_{G}^{-1} (0) \pitchfork s^{-1}(0)$, and $G$ acts freely on $\mu_{G}^{-1} (0)$;
		\item $M$ and $M \red G$ are circle compact (see Definition~\ref{def:cc}) for their respective $S^{1}$-actions.
	\end{enumerate}
	Then the restriction map
	\begin{equation}
		H^{*} (M \red G; R) \longrightarrow H^{*} (M \rred G; R)
	\end{equation}
	is an isomorphism for every coefficient ring $R$.
\end{theorem}

Combining these, we obtain the main result of this paper.

\begin{theorem}[Cf. Theorem~\ref{thm:main}]
	\label{thm:main_intro}
	When the assumptions of both Theorems~\ref{thm:kirwan_surj_intro} and~\ref{thm:eq_sommese_intro} are satisfied, the hyper-Kirwan map is surjective.
\end{theorem}

Our theorem applies in particular to hypertoric varieties \cite{BielawskiDancer, KonnoToric} and hyperpolygon spaces \cite{KonnoPolygon, FisherRayan}.

\begin{remark}
	Independently, McGerty and Nevins have obtained a surjectivity result for a large class of algebraic symplectic quotients, based on their results on the derived categories of quotient stacks~\cite{McGertyNevinsSurj}.
	
	Our results are also related to those of~\cite{Wilkin19}.
\end{remark}

\subsection{Hypertoric case}
The setup above includes the case of hypertoric manifolds, i.e. hyperkähler manifolds with a trihamiltonian action of a compact abelian group $T$, in the following way.

Recall that a hyperkähler manifold is a Riemannian manifold $(M,g)$ equipped with three compatible complex structures $(\bi, \bj, \bk)$ satisfying the quaternion relations.
We denote by $(\omega_{1}, \omega_{2}, \omega_{3})$ the respective symplectic (Kähler) forms.
Focusing on the particular complex structure $\bi$, it is common to denote by $\omegar \coloneqq \omega_{1}$ and $\omegac \coloneqq \omega_{2} + \sqrt{-1} \omega_{3}$ the \emph{real} and \emph{complex} (or \emph{holomorphic}) symplectic forms, respectively.

If a Lie group $G$ acts on $M$, we say the action is \emph{trihamiltonian} if there exists a \emph{hyperkähler moment map}, i.e. a function
\begin{equation}
	\mu_{G} = (\mu_{1}, \mu_{2}, \mu_{3}) \colon M \to \fgd \otimes \BR^{3} \, ,
\end{equation}
where $\fg$ denotes the Lie algebra of $G$, such that each $\mu_{i}$ is an ordinary moment map with respect to $\omega_{i}$.
Again, it is customary to call $\mur \coloneqq \mu_{1}$ and $\muc \coloneqq \mu_{2} + \sqrt{-1} \mu_{3}$ the \emph{real} and \emph{complex} (or \emph{holomorphic}) moment maps.
It is easy to see that $\muc \colon M \to \fgd_{\BC}$ is indeed holomorphic with respect to the complex structure $\bi$, and so its level sets are complex analytic subvarieties.

Given such data, the hyperkähler quotient construction~\cite{HKLR} is a quaternionic (or holomorphic symplectic) analogue of the Marsden-Weinstein reduction.
Assuming that $0$ is a regular value for the real moment map, one can note that $(0, 0, 0)$ is regular for the hyperkähler moment map.
The Kähler and hyperkähler quotients of $M$ are then defined, respectively, to be
\begin{equation}
	\begin{aligned}
		M \red G &\coloneqq \mur^{-1}(0) / G \, , \\
		M \rred G &\coloneqq \mu_{G}^{-1}(0,0,0) / G = \bigl(\mur^{-1}(0) \cap \muc^{-1}(0) \bigr) / G \, .
	\end{aligned}
\end{equation}
Although $M \rred G$ is naturally hyperkähler, if $G = T$ is abelian and one only considers the orbifold or Kähler structure on the reduction, then what we just described is a particular case of the hyper-reduction of eq.~\eqref{eq:reduction} for $s = \mu_{\BC}$.
Our theory may then be applied once an additional $S^{1}$-action with the required properties is given---as in the setting of~\cite{FisherRayan}, the latter will be assumed to act by hyperkähler rotations rather than fixing the entire structure, as long as it is Hamiltonian for the Kähler form it preserves.

\subsection*{Plan of the exposition}

Our paper is structured as follows.
In Section~\ref{sec:kirsurj}, we will discuss the notion of circle compactness and the relevant associated concepts, and prove Theorem~\ref{thm:kirwan_surj_intro}.
In Section~\ref{sec:restrictions}, after introducing more definitions, we will show several results relating the cohomology of manifolds with appropriate structure to that of subspaces obtained as vanishing loci of functions or sections, including Theorems~\ref{thm:eq_sommese_intro} and~\ref{thm:main_intro}.
In the final Section~\ref{sec:hypertoric} we focus on the case of a hypertoric manifolds and show an extended version of Theorem~\ref{thm:main_intro} in that case.

\subsection*{Acknowledgements}
All authors wish to acknowledge the contributions of Young-Hoon Kiem, Frances Kirwan, and Jon Woolf, who were co-authors of an earlier draft of this manuscript.
Many of the ideas in the present version originated in that earlier draft.
We would also like to acknowledge Tam\'as Hausel for carrying out a detailed analysis which provided a counterexample to an argument stated in that earlier version of this manuscript, which was written by JF and LJ together with other authors.
JF and LJ would like to thank Peter Crooks for his reading of said prior version.
AM and SR are grateful to JF and LJ for the invitation to join the revised manuscript and for the ensuing collaborative discussions.
We also thank Jon Woolf for helpful comments at one point during the preparation of the revised manuscript.
JF acknowledges an affiliation with, and support from, the University of Hamburg in 2014 during a postdoctoral fellowship in which some of the material in this manuscript was developed.
LJ is partially supported by an NSERC Discovery Grant.
AM acknowledges generous funding from LJ's Discovery Grant as well as that of Marco Gualtieri.
SR is partially supported by an NSERC Discovery Grant, a Pacific Institute for the Mathematical Sciences (PIMS) Collaborative Research Group (CRG) grant, and a Tri-Agency New Frontiers in Research (NFRF) Exploration Stream grant.
Finally, we thank the anonymous referee for a careful reading and useful suggestions.

\renewcommand{\thetheorem}{\arabic{section}.\arabic{theorem}}

\section{Kirwan surjectivity for circle compact manifolds}
\label{sec:kirsurj}

Throughout this section, $M$ will denote a symplectic manifold with a fixed Hamiltonian $S^{1}$-action.
We emphasise again that, wherever unspecified and unless otherwise noted, all cohomology groups are taken with respect to rational coefficients.

\begin{definition}
	\label{def:cc}
	$M$ is said to be \emph{circle compact}%
	\footnote{Also called \emph{semiprojective}~\cite{Hausel2013} in the Kähler case.}
	if the $S^{1}$-action satisfies the following conditions:
	\begin{enumerate}
		\item The fixed point set $M^{S^{1}}$ is compact;
		\item The $S^{1}$-moment map $\mu_{S^{1}} \colon M \to \BR$ is proper and bounded below.
	\end{enumerate}
\end{definition}

Given a circle compact manifold $M$, we may construct a compactification $\overline{M}$ as follows~\cite{Lerman, Hausel1998}.
Consider on $M \times \BC$ the moment map $\widetilde{\mu}_{S^{1}}$ of the diagonal $S^{1}$-action, and fix $c \in \BR$ larger than all the critical values of $\mu_{S^{1}}$.
One sees that
\begin{equation}
	\widetilde{\mu}_{S^{1}}^{-1} (c) = \bigl(M^{<c} \times S^{1} \bigr) \cup \mu_{S^{1}}^{-1} (c) \, ,
\end{equation}
where $M^{<c} \coloneqq \mu_{S^{1}}^{-1} \bigl((-\infty,c)\bigr)$ is, by elementary Morse theory with the moment map, diffeomorphic to $M$ through an equivariant map identifying $\mu_{S^{1}}$ with its restriction.
Since $c$ is regular for both $\mu_{S^{1}}$ and $\widetilde{\mu}_{S^{1}}$, one can write
\begin{equation}
	(M \times \BC) \reda{c} S^{1} = M^{<c} \cup \bigl( M \reda{c} S^{1} \bigr) \, ,
\end{equation}
and it follows from the assumptions on $\mu_{S^{1}}$ that the above is compact, and again by Morse theory different choices of $c$ produce diffeomorphic orbifolds.
Although $c$ does affect the symplectomorphism classes of $M^{<c}$ and $\overline{M}$, as one may note by considering their volumes, we shall neglect this dependence as it is of no consequence for our arguments, and henceforth often identify $M^{<c} \subseteq \overline{M}$ with $M$.
There is also an additional Hamiltonian $S^{1}$-action on $\overline{M}$ descending from those on the first factor of $M \times \BC$, and both restrict the existing structure on $M$.

\begin{definition}
	\label{def:cut_compactification}The \emph{cut compactification} of a circle compact symplectic manifold $M$ is the orbifold
	\begin{equation}
		\overline{M} \coloneqq (M \times \BC) \reda{c} S^{1} \, .
	\end{equation}
	The \emph{boundary divisor} of $\overline{M}$ is
	\begin{equation}
		D_{M} \coloneqq \overline{M} \setminus M \iso M \reda{c} S^{1} \, .
	\end{equation}
\end{definition}

The term ``boundary divisor", while a slight abuse of notation in this context, is inspired by the case where $M$ is Kähler.

\begin{lemma}
	\label{lemma-restriction-surjectivity}
	If $M$ is circle compact and $S^{1}$ acts freely on $\mu_{S^{1}}^{-1} (c)$, then the natural restriction $H^\ast(\overline{M}) \to H^\ast(M)$ is surjective.
\end{lemma}

\begin{proof}
	Under our assumptions, both $\overline{M}$ and $D_{M}$ are smooth, the latter an embedded submanifold.
	The boundary divisor is, by construction, the absolute maximum locus of the moment map on $\overline{M}$.
	Therefore, it is a critical component of index equal to its codimension, i.e. $2$, while all other components are contained in $M$.

	According to Proposition~5.8 of~\cite{Kirwan} and Remark/Paragraph 5.9 therein, the moment map of an $S^{1}$-action on a compact symplectic manifold is perfect as a Morse-Bott function.
	In her proof, Kirwan establishes the converse of the Morse inequalities by inspecting the Serre spectral sequence of a homotopy quotient, an argument that does not rely on compactness.
	The conclusion then holds whenever Morse inequalities can be established, which is the case for $M$ since $\mu_{S^{1}}$ is proper and bounded below.
	We may then conclude that
	\begin{equation}
		P_{t} (\overline{M}) = P_{t}(M) + t^{2} P_{t} (D_{M}) \, , 
	\end{equation}
	where $P_{t}$ denotes Poincaré polynomial in the variable $t$.
	Therefore, by induction, the Thom-Gysin sequence of $D_{M} \subseteq \overline{M}$ splits into short exact portions
	\begin{equation}
		0  \longrightarrow H^{\ast-2} (D_M) \longrightarrow  H^\ast(\overline{M}) \longrightarrow H^\ast(M) \longrightarrow 0 \, ,
	\end{equation}
	whence the result.
\end{proof}

We now turn our attention to the case where $M$ is equipped with an additional action of a compact Lie group $G$.
Suppose the overall $G \times S^{1}$-action is Hamiltonian, so in particular the moment map $\mu = (\mu_{G}, \mu_{S^{1}})$ is equivariant, and that $0$ is a regular value for $\mu_{G}$.
In that case, the $S^{1}$-action and its Hamiltonian function descend to the reduction $M \red G$, and it makes sense to talk about circle compactness of the latter space.

\begin{theorem}
	\label{thm:KirwanSurj}
	Suppose $G$ is a compact Lie group, $(M, \omega)$ a Hamiltonian $G \times S^{1}$-space with $0 \in \fgd$ as a regular value of the $G$-moment map.
	If both $M$ and $M \red G$ are circle compact, the latter with smooth compactification, then the Kirwan map $\kappa \colon H^{*}_{G} (M) \to H^{*}(M \red G)$ is surjective.
\end{theorem}

\begin{proof}
	Let $\overline{M \red G}$ denote the cut compactification of $M \red G$ as in Definition~\ref{def:cut_compactification}.
	By construction, the inclusion map $M \red G \hookrightarrow \overline{M \red G}$ can be regarded as
	\begin{equation}
		\bigl( M \times \BC^{*} \bigr) \red \bigl( G \times S^{1} \bigr) \hookrightarrow \bigl( M \times \BC \bigr) \red \bigl( G \times S^{1} \bigr) \, ,
	\end{equation}
	from which we obtain a commutative diagram
	\begin{equation} \label{diagram-equivariant-surjectivity} \nonumber
	  \begin{tikzpicture}[baseline=(current bounding box.center)]
	    \matrix (m) [matrix of math nodes,row sep=3em,column sep=3em,minimum width=2em]
	    {
	      H^\ast_{G \times S^1}(M \times \BC) & H^\ast(\overline{M \red G}) \\
	      H^\ast_{G \times S^1}(M \times \BC^\ast) & H^\ast(M \red G) \\
	    };
	    
	    \path[-stealth] (m-1-1) edge (m-1-2);
	    \path[-stealth] (m-2-1) edge (m-2-2);
	    \path[-stealth] (m-1-1) edge (m-2-1);
	    \path[-stealth] (m-1-2) edge (m-2-2);
	
	  \end{tikzpicture}
	\end{equation}
	Since $M$ is circle compact, in particular the Hamiltonian of the $S^{1}$-action is proper, and therefore so is the full $G \times S^{1}$-moment map on $M \times \BC$.
	Therefore, the top horizontal map in the diagram is surjective by usual Atiyah-Bott-Kirwan theory.
	Since the right vertical arrow is surjective by Lemma~\ref{lemma-restriction-surjectivity}, we may conclude that the bottom horizontal map is also onto.
	On the other hand, the $S^{1}$-action on $M \times \BC^{*}$ is free, so $H^{*}_{G \times S^{1}} (M \times \BC^{*}) \iso H^{*}_{G} (M)$.
	Thus $H^{*}_{G} (M) \to H^{*} (M \red G)$ is surjective.
\end{proof}

That concludes our results for symplectic manifolds; starting with the next section, we will be requiring additional structure on $M$.

\section{Lefschetz-Sommese theorems for actions of semi-linear type}
\label{sec:restrictions}

The next step in our program is to study the effect induced on cohomology, both ordinary and equivariant, by the operation of restricting to vanishing loci of appropriate functions or sections.
This is reminiscent of the Sommese Theorem and the Lefschetz Hyperplane Theorem~\cite{Sommese1977,Lefschetz24,Bott,Lazarsfeld2004II}.
For later reference, let us recall their precise statements.

\begin{theorem}[Sommese~\cite{Sommese1977,Lazarsfeld2004II}]
	\label{thm:Sommese}
	Suppose $X$ is a non-singular projective variety of dimension $n$, $E \to X$ an ample vector bundle of rank $e$.
	If $\sigma$ is a section of $E$ and $Z = \sigma^{-1} (0)$, then
	\begin{equation}
		H^{i} (X, Z; \BZ) = 0 \qquad \text{for $i \leq n-e$.}
	\end{equation}
	In particular, the restriction map $H^{i} (X; \BZ) \to H^{i} (Z; \BZ)$ is an isomorphism for $i < n - e$ and an embedding for $i = n - e$.
\end{theorem}

\begin{theorem}[Lefschetz Hyperplane~\cite{Bott}]
	\label{thm:LHT}
	Suppose $X$ is a compact complex manifold of dimension $n$, $E \to X$ a positive holomorphic line bundle.
	If $\sigma$ is a section of $E$ intersecting $0$ transversely and $Z = \sigma^{-1} (0)$, then
	\begin{equation}
		H^{i} (X, Z; R) = 0 \qquad \text{for $i \leq n-1$}
	\end{equation}
	for every coefficient ring $R$, and the restriction map $H^{i} (X; R) \to H^{i} (Z; R)$ is an isomorphism for $i < n - 1$ and an embedding for $i = n - 1$.
\end{theorem}

Although the compactness condition is generally broken in our setting, we shall overcome this difficulty by using cut compactifications, working equivariantly, or both.
Furthermore, it turns out that almost complex structures are enough for some of our results, without requiring integrability.

Before we start, a few definitions are in order.
Henceforth we will assume $(M,\omega)$ to be endowed with a compatible almost complex structure $I$ and a Hamiltonian isometric $S^{1}$-action.
In this case, if the manifold is circle compact then the cut compactification $\overline{M}$ is itself an almost Kähler manifold, and if $I$ is integrable then $D_{M}$ becomes a divisor in the usual sense.

We will also assume $M$ to be decorated in the following sense.

\begin{definition}
	\label{def:decorated}
	We call $M$ a \emph{decorated manifold} if it is endowed with a holomorphic function $s$ valued in a vector space $V$, homogeneous of some positive degree $d$ under the $S^{1}$-action, which is to say that
	\begin{equation}
		s(tx) = t^{d} s(x)
	\end{equation}
	for all $t \in S^{1}$ and $x \in M$.
	If $G$ is a compact Lie group, we will say that $M$ is a \emph{decorated $G$-manifold} if it is equipped with a Hamiltonian $G \times S^{1}$-action preserving $s$ and the almost Kähler structure.
	We will denote $\mu = (\mu_{G}, \mu_{S^{1}})$ the corresponding moment map.
\end{definition}

Note that, if $M$ is circle compact under the $S^{1}$-action, $s$ does not descend to a function on the cut compactification, but rather to a section of an appropriate vector bundle $E$.
Upon further inspection, when the terminology makes sense $E$ is isomorphic to $\underline{V} \otimes \CO_{\overline{M}} (D_{M})^{d}$, where $\underline{V}$ is the trivial vector bundle with fibre $V$.

\begin{definition}
	\label{def:regular_decorated}
	We will say that a decorated $G$-manifold is
	\begin{itemize}
		\item \emph{$s$-regular} if $0 \in V$ is a regular value for $s$;
		\item \emph{$G$-regular} if $0 \in \fgd$ is a regular value for $\mu_{G}$ and the $G$-action is free on $\mu_{G}^{-1} (0)$;
		\item \emph{hyper-regular} if it is both $s$- and $G$-regular, and moreover $\mu^{-1}_{G}(0) \pitchfork s^{-1} (0)$.
	\end{itemize}
	If $M$ is $G$-regular, we will denote $\check{\mu}_{S^{1}}$ and $\check{s}$ the induced $S^{1}$ moment map and holomorphic function on the almost Kähler manifold $M \red G$.
	If it is hyper-regular, we will call
	\begin{equation}
		M \rred G \coloneqq \bigl( \mu_{G}^{-1} (0) \cap s^{-1} (0) \bigr) / G \, .
	\end{equation}
	the \emph{hyper-reduction} of $M$.
\end{definition}

As a degenerate case, every decorated manifold is $G$-regular for the action of the trivial group $G = \{1\}$, hyper-regular if and only if $s$-regular.

\begin{definition}
	\label{def:linear_type}
	We will say that the $G \times S^{1}$-action (or simply the $G$-action for brevity) on a $G$-regular decorated $G$-manifold is
	\begin{itemize}
		\item of \emph{semi-linear type} if both $M$ and $M \red G$ are circle compact for their respective $S^{1}$-actions;
		\item of \emph{linear type} if moreover the almost complex structure $I$ is integrable, the circle compactification $\overline{M}$ is smooth, and the line bundle $\CO_{\overline{M}} (D_{M})$ is ample.
	\end{itemize}
\end{definition}

In the degenerate case of $G = \{1\}$, the condition of semi-linear type reduces to circle compactness of $M$.
We also stress that, in the hyper-regular semi-linear case, one can show that $M \rred G$ is circle compact by using that it sits inside $M \red G$ as a closed submanifold.
Furthermore, if $M$ is a circle compact Kähler manifold, $D_{M}$ smooth, and $\CO_{\overline{M}} (D_{M})$ ample, then $\overline{M}$ is a projective variety, and the same goes for $\overline{M \red G}$ and $\overline{M \rred G}$ given the appropriate data.

For some of the following applications, one or more of the conditions in these definitions may be relaxed.
For instance, in some cases it is enough to assume that $M$ is acted on by a central extension of $G$ by $S^{1}$, not necessarily trivial, with due adjustments to the equivariance properties of $s$.
In others, the regularity conditions on $\mu_{G}$ and $s$ may be weakened.
In fact, one could alternatively require that $0 \in \fgd \oplus V$ be a regular value for $\mu_{G} \oplus s$ and that $G$ act freely on $(\mu_{G} \oplus s)^{-1} (0)$.
This alone would be enough to guarantee that the hyper-reduction is a well defined smooth manifold carrying a natural almost Kähler structure, but in our assumptions we also have smoothness of $Z \coloneqq s^{-1} (0)$ and $M \red G$.
The transversality condition then implies that $0$ is a regular value for both $\mu_{G} \vert_{Z}$ and $\check{s}$, so $M \rred G$ can be realised equivalently as either $Z \red G$ or $\check{s}^{-1} (0)$.
Regardless, we shall adhere to the definitions given above, for the sake of simplicity.

The following result will serve as a justification for our definition.

\begin{theorem}
	Suppose $G$ is acting linearly on $\BC^{n}$ and consider the naturally induced tri-hamiltonian action on $M \coloneqq T^{*} \BC^{n}$.
	Consider $M$ as a Kähler manifold with the complex structure corresponding to the identification with $\BC^{2n}$, and for $V \coloneqq \fgd_{\BC}$ let $s \colon M \to V$ be the complex moment map of the $G$-action.
	Under the natural $S^{1}$-action on $\BC^{2n}$, $s$ is $2$-homogeneous, $M$ is circle compact, the compactification $\overline{M}$ is smooth, and the line bundle $\CO_{\overline{M}} (D_{M})$ is ample.
	Furthermore, if $\gamma \in \fgd$ is a regular value for $\mu_{G}$ and $M \reda{\gamma} G$ is smooth, then the latter is also circle compact.
\end{theorem}

In other words, $M$ is a decorated $G$-manifold, and up to the regularity conditions the action is of linear type.

\begin{proof}
	Since the origin is the only fixed point for the $S^{1}$-action, and its moment map is $\mu_{S^{1}} (x) = \frac{1}{2} \lVert x \rVert^{2}$, $M$ is clearly circle compact.
	The cut compactification is $\overline{M} \iso \BP^{2n}$ with the Fubini-Study Kähler structure (possibly up to a rescaling) and boundary divisor $D \simeq \BP^{2n-1}$, the projective space at infinity.
	The line bundle $\CO_{\overline{M}} (D_{M})$ is then isomorphic to $\CO_{\BP^{2n}} (1)$, which is ample.
	
	Using coordinates $(\mathbf{x}, \mathbf{y})$ on $M \iso \BC^{n} \oplus \BC^{n}$, the $S^{1}$-action reads $t.(\mathbf{x}, \mathbf{y}) = (t \mathbf{x}, t \mathbf{y})$.
	On the other hand, the complex symplectic form is
	\begin{equation}
		\omegac = \sum_{i = 1}^{n} \dif x_{i} \wedge \dif y_{i} \, ,
	\end{equation}
	so that clearly $t^{*} \omegac = t^{2} \omegac$.
	It follows that $s$ is also $2$-homogeneous as claimed, being determined by $\omegac$ and the $S^{1}$-invariant $\fg$-action.
	
	The only remaining condition to check is that $(M \reda{\gamma} G)^{S^{1}}$ is compact, since properness and boundedness descend immediately from $M$.
	By properness, this is equivalent to showing that the fixed point set has only finitely many connected components.
	However, Morse theory with $\mu_{S^{1}}$ implies that if there were infinitely many connected components then $H^{i} (M \reda{\gamma} G)$ would have infinite rank for some $i$.
	On the other hand, ordinary Kirwan surjectivity implies that $H^{*} (M \reda{\gamma} G)$ is a quotient of $H^{*} (BG)$, so it is finite-dimensional in every degree.
\end{proof}

\subsection{Equivariant Lefschetz-Sommese theorem}

We are now ready to state our first result about restriction maps.

\begin{theorem}
	\label{thm:equivariant_sommese}
	If $M$ is a circle compact $s$-regular decorated $G$-manifold, then the inclusion $s^{-1}(0) \hookrightarrow M$ induces an isomorphism
	\begin{equation}
		H^{*}_{K} (M; R) \to H^{*}_{K} (s^{-1} (0); R)
	\end{equation}
	for every subgroup $K < G \times S^{1}$ and every coefficient ring $R$.
\end{theorem}

In particular, choosing $K = \{1\}$ or $K = S^{1}$ gives equivalences in ordinary and $S^{1}$-equivariant cohomology.
Note that, unlike the Sommese Theorem and the Lefschetz Hyperplane Theorem, our result states that the restriction map is an isomorphism in all degrees, something which in the projective case would be obstructed by Poincaré duality.
In our setting, however, $M$ is necessarily non-compact, otherwise $s$ would be constant and therefore vanish everywhere by $S^{1}$-homogeneity, which would violate the $s$-regularity condition.

\begin{proof}
	Throughout this proof, we shall call $Z \coloneqq s^{-1} (0)$ for simplicity.

	Notice first of all that $H^{*}_{K} (Z; R)$ makes sense since $Z$ is preserved by the $G \times S^{1}$-action.
	The proof will use Morse-Bott theory with the moment map $\mu_S^{1}$, and it will be structured as follows.
	First, we will show that $M^{S^{1}} \subseteq Z$, i.e. the critical locus of $\mu_{S^{1}}$ is contained in $Z$.
	Second, we will argue that, at every fixed point $x \in M^{S^{1}}$, the restriction of the isotropy $S^{1}$-action to the normal space $N_{x} Z$ has positive weights, and therefore, the indices of $\mu_{S^{1}}$ are the same whether computed on $M$ or on $Z$.
	Next, we will discuss how, for every critical component $\CC$ of $\mu_{S^{1}}$, the immersions of the associated disc bundle $D\CC$ into $M$ and $Z$ commute with the inclusion $Z \hookrightarrow M$.
	Moreover, we will see that if $M_{\CC}^{\pm}$ and $Z_{\CC}^{\pm}$ denote the associated super- and sub-level sets then the homotopy equivalences $M_{\CC}^{+} \iso M_{\CC}^{-} \cup D\CC$ and $Z_{\CC}^{+} \iso Z_{\CC}^{-} \cup D\CC$ are $(G \times S^{1})$-equivariant.
	In the final part of the proof, we will tie the pieces together and use induction to conclude that the inclusion maps induce isomorphisms at all levels in the Thom-Gysin sequences arising from the Morse stratification.
	
	For the first step of our proof, note that for every $x \in M^{S^{1}}$ and $t \in S^{1}$ we have that
	\begin{equation}
		s(x) = s(tx) = t^{d} s(x)
	\end{equation}
	since $s$ is $d$-homogeneous for $S^{1}$, and therefore $s(x) = 0$, i.e. $x \in Z$.
	
	For the second part, suppose that $x \in M^{S^{1}}$.
	The induced isotropy actions on $T_{x} M$ and $T_{x} Z$ are then linear with respect to $I$, so we obtain a short exact sequence of complex $S^{1}$-representations
	\begin{equation}
		0 \longrightarrow T_{x} Z \longrightarrow T_{x} M \longrightarrow N_{x} Z \longrightarrow 0 \, .
	\end{equation}
	Notice now that if $\xi \in T_{x} M$ is a vector of weight $w$ for the isotropy action, then
	\begin{equation}
		t^{w} \dif[x] s (\xi) = \dif[x] s (t^{w} \xi) = \dif[x] s (t_{*} \xi) = \bigl(t^{*} \dif[x] s \bigr) (\xi) = \bigl(\dif[x] (t^{*} s)\bigr) (\xi) = t^{d} \dif[x] s(\xi) \, ,
	\end{equation}
	so $w = d$ unless $\dif[x] s(\xi) = 0$.
	Since $d$ is positive and the weights of this action are the eigenvalues of the Hessian of $\mu_{S^{1}}$, it follows that the latter is positive definite on $N_{x} Z$.
	Thus, the negative subspace of $T_{x} M$ must sit inside $T_{x} Z$, and therefore $\mu_{S^{1}}$ has the same index on $T_{x} Z$ as on $T_{x} M$.
	
	We now turn our attention to our third claim, which is that the attachment maps from the Morse stratifications of $M$ and $Z$ are compatible with the inclusion---to that end we consider the gradient flows defined by $\mu_{S^{1}}$ on $M$ and $Z$.
	On the one hand, at every point $x \in Z$ the gradient vector $\nabla_{x}^{Z} \mu_{S^{1}}$ is a priori the orthogonal projection of $\nabla_{x}^{M} \mu_{S^{1}}$ onto $T_{x} Z$.
	On the other hand, $\nabla^{M} \mu_{S^{1}}$ can be written as $I \xi_{S^{1}}$, where $\xi_{S^{1}}$ is the vector field generating the $S^{1}$-action.
	Since $Z$ is an almost complex submanifold preserved by $S^{1}$, it follows that $\nabla_{x}^{M} \mu_{S^{1}}$ is tangent to $Z$, and therefore $\nabla_{x}^{Z} \mu_{S^{1}} = \nabla_{x}^{M} \mu_{S^{1}}$.
	Moreover, by the circle compactness assumption $\mu_{S^{1}}$ is bounded below and proper on $M$, and therefore the steepest descent flow is forward-complete.
	We then obtain Morse stratifications on $M$ and $Z$, and for every critical component $\CC$ the gradient flow lines emanating from $\CC$ lie in $Z$ at all times.
	This shows that the embeddings of the corresponding disc bundle $D\CC$ into $M$ and $Z$ are compatible with the inclusion $Z \hookrightarrow M$, as claimed.
	To verify our statement about equivariance of the homotopy equivalences, notice that, for every element $a \in \fg \oplus \BR$ we have
	\begin{equation}
		[a^{\#}, \nabla \mu_{S^{1}}] = [a^{\#}, I \xi_{S^{1}}] = \bigl( \CL_{a^{\#}} I \bigr) \xi_{S^{1}} + I [a^{\#}, \xi_{S^{1}}] = 0 \, ,
	\end{equation}
	since $S^{1}$ is central and $G \times S^{1}$ preserves $I$.
	Thus the $(G \times S^{1})$-action commutes with the gradient flow, which is what defines the homotopy equivalences in question.
	
	We are now ready to approach the final part of our proof.
	For each critical component $\CC$ we may consider the pairs $(M^{+}_{\CC}, M^{-}_{\CC})$ and $(Z^{+}_{\CC}, Z^{-}_{\CC})$, their respective long exact sequences, and the mapping between them induced by the inclusion maps $Z_{\CC}^{\pm} \hookrightarrow M_{\CC}^{\pm}$.
	If $D\CC$ denotes the unit disc bundle in the negative normal bundle of $\CC$, the cohomology of each pair can be replaced by that of $(D\CC, \partial D\CC)$ via the usual excision argument.
	We then obtain the following as a portion of the resulting commutative diagram:
	\begin{equation*}
		\begin{tikzpicture}[baseline=(current bounding box.center)]
	    \matrix (m) [matrix of math nodes,row sep=3em,column sep=1.2em,minimum width=2em]
	    {
	     H_{K}^{*-1}(M^{-}_{\CC})  & H_{K}^{*-1} (D\CC, \partial D\CC) & H_{K}^{*}(M^{+}_{\CC})  & H_{K}^{*}(M^{-}_{\CC}) &  H_{K}^{*}(D\CC, \partial D\CC) \\
	     H_{K}^{*-1}(Z^{-}_{\CC})  & H_{K}^{*-1} (D\CC, \partial D\CC) & H_{K}^{*}(Z^{+}_{\CC})  & H_{K}^{*}(Z^{-}_{\CC}) &  H_{K}^{*}(D\CC, \partial D\CC) \\
	    };
	
	    \foreach \n [count=\np from 2] in {1,2,3,4}
	    {
	      \path[-stealth] (m-1-\n) edge (m-1-\np);
	      \path[-stealth] (m-2-\n) edge (m-2-\np);
	    };
	
	    \foreach \n in {1,2,3,4,5}
	    {
	      \path[-stealth] (m-1-\n) edge (m-2-\n);
	    };
	
	  \end{tikzpicture}
	\end{equation*}
	$H^{*}_{K}$ is understood to be taken with coefficients in $R$.
	Since $\mu_{S^{1}}$ is proper and bounded below, it has a critical component $\CC_{0}$, for which $M^{-}_{\CC_{0}} = Z^{-}_{\CC_{0}} = \emptyset$.
	The diagram then splits into a family of squares and we may conclude that $\iota^{*} \colon H^{*}_{K} (M^{+}_{\CC_{0}}) \to H^{*}_{K} (Z^{+}_{\CC_{0}})$ is an isomorphism.
	On the other hand, if for some critical component $\CC$ the first and fourth vertical maps are isomorphisms, then so is the third by the Five Lemma.
	Our conclusion then follows by induction on the critical sets, which are finitely many since the fixed locus $M^{S^{1}}$ is compact under the circle compactness assumption.
\end{proof}

\begin{corollary}
	If $M$ is a circle compact $s$-regular decorated $G$-manifold and the $G$-action is locally free on $s^{-1} (0)$, then the natural map
	\begin{equation}
		H^{*}_{G \times S^{1}} (M; R) \to H^{*}_{S^{1}} \bigl( s^{-1} (0) \slash G; R \bigr)
	\end{equation}
	is an isomorphism for every coefficient ring $R$.
\end{corollary}

\begin{proof}
	By Theorem~\ref{thm:equivariant_sommese}, the restriction map $H^{*}_{G \times S^{1}} (M; R) \to H^{*}_{G \times S^{1}} \bigl( s^{-1} (0); R \bigr)$ is an isomorphism.
	On the other hand, if the $G$-action on $s^{-1} (0)$ is locally free, then
	\begin{equation}
		H^{*}_{G \times S^{1}} (s^{-1} (0); R) \iso H^{*}_{S^{1}} \bigl( s^{-1} (0) \slash G; R \bigr) \, ,
	\end{equation}
	showing our claim.
\end{proof}

\begin{theorem}
	\label{thm:equivariant_sommese_quotient}
	Suppose $M$ is a hyper-regular decorated $G$-manifold and the action is of semi-nearnear type.
	Then the maps
	\begin{equation}
		H^{*} (M \red G ; R) \to H^{*} (M \rred G ; R) \,
		\qquad \text{and} \qquad
		H^{*}_{S^{1}} (M \red G ; R) \to H^{*}_{S^{1}} (M \rred G ; R)
	\end{equation}
	induced by the inclusion $M \rred G \hookrightarrow M \red G$ are isomorphisms for every coefficient ring $R$.
\end{theorem}

\begin{proof}
	For convenience, call $Y = M \red G$ and $Z = \check{s}^{-1}(0) = M \rred G$.
	By the regularity assumption, $Y$ is smooth, and for every $y = [x] \in Y$ the space $T_{y} Y$ can be realised as the orthogonal complement of $T_{x} (G x)$ in $T_{x} \bigl(\mu_{G}^{-1} (0)\bigr)$ with respect to the Riemannian metric $g$.
	On the other hand, $T_{x} \bigl(\mu_{G}^{-1} (0) \bigr)$ is the orthogonal complement of $I T_{x} (Gx)$ in $T_{x} M$, so $T_{y} Y$ can equivalently be expressed as the orthogonal in $T_{x} M$ of $\BC T_{x} (Gx)$.
	
	Under the identification above, $\dif[y] \check{s}$ is the restriction of $\dif[x] s$.
	Notice now that if $s(x) = 0$ then $\dif[x] s \colon T_{x} M \to V$ is surjective, since $0 \in V$ is a regular value of $s$.
	On the other hand, $G$-invariance of $s$ implies that $\dif[x] s$ vanishes on $T_{x} (Gx)$, and therefore on $\BC T_{x} (Gx)$ by $\BC$-linearity.
	Therefore, $\dif[x] s$ remains surjective when restricted to the orthogonal complement of $\BC T_{x} (Gx)$ in $T_{x} M$, which is to say that $\dif[y] \check{s}$ is surjective on $T_{y} Y$ if $\check{s} (y) = 0$.
	In summary, $0 \in V$ is a regular value for $\check{s}$, which is clearly homogeneous for the $S^{1}$-action with the same degree as $s$.
	In summary, $Y$ itself satisfies the hypotheses of Theorem~\ref{thm:equivariant_sommese}, and the conclusion follows.
\end{proof}

The following summarises the main surjectivity result as presented in Theorem~\ref{thm:main_intro}.

\begin{theorem}
	\label{thm:main}
	If $M$ is a hyper-regular decorated $G$-manifold with an action of semi-linear type, then the induced maps
	\begin{equation}
		H^{*}_{G} (M) \to H^{*} (M \rred G)
		\qquad \text{and} \qquad
		H^{*}_{G} (s^{-1}(0)) \to H^{*} (s^{-1}(0) \red G)
	\end{equation}
	in \emph{rational} cohomology are surjective.
\end{theorem}

\begin{proof}
	The first map is the hyper-Kirwan map $\kappa_{G,s}$, and as noted in the introduction it splits as the composition
	\begin{equation}
		H^{*}_{G} (M) \stackrel{\kappa}{\rightarrow} H^{*} (M \red G) \stackrel{\iota^{*}}{\rightarrow} H^{*} (M \rred G) \, .
	\end{equation}
	We have established that $\kappa$ is surjective in Theorem~\ref{thm:KirwanSurj} and that $\iota^{*}$ is an isomorphism in Theorem~\ref{thm:equivariant_sommese_quotient}.
	We have also shown with Theorem~\ref{thm:equivariant_sommese} that $H^{*}_{G} (M) \iso H^{*}_{G} (s^{-1}(0))$, and since $s^{-1}(0) \red G = M \rred G$ the map $H^{*}_{G} (s^{-1}(0)) \to H^{*} (s^{-1}(0) \red G)$ is essentially the same as the hyper-Kirwan map.
\end{proof}

\subsection{Additional results on the cut compactifications}

Throughout this section we have thus far only invoked the properties of the $S^{1}$-actions under the circle compactness assumption without actually involving the cut compactifications of $M$ and its reductions.
The properties of these spaces are the focus of the remaining theorems in this section, which are direct consequences of the Sommese Theorem~\cite{Sommese1977} and the Lefschetz Hyperplane Theorem~\cite{Bott}.

\begin{theorem}
	\label{thm:newSommese}
	Suppose $M$ is a decorated manifold and the $S^{1}$-action is of linear type.
	Call $\sigma \in \Gamma(\overline{M}, E)$ the section induced by $s$ and $k \coloneqq \dim_{\BC} M - \dim_{\BC} V$.
	Then the restriction map
	\begin{equation}
		H^{i} \bigl( \overline{M} \bigr) \longrightarrow H^{i} \bigl( \sigma^{-1} (0) \bigr)
	\end{equation}
	is an isomorphism for $i < k$ and injective for $i = k$.
\end{theorem}

As in the Sommese theorem, no particular regularity assumptions on $s$ are strictly required, and the locus $\sigma^{-1} (0)$ may be singular.

\begin{proof}
	Recall that, according to Definition~\ref{def:linear_type}, the action being of linear type (and not just semi-linear) implies that the complex structure on $M$ is integrable and the boundary divisor defines an ample line bundle, which makes $\overline{M}$ a projective variety.
	Furthermore, $E$ is isomorphic to $\underline{V} \otimes \CO_{\overline{M}}(D_{M})^{\otimes d}$, which is itself ample.
	Our conclusion follows then from the ordinary Sommese Theorem (see Theorem~\ref{thm:Sommese}).
\end{proof}

\begin{theorem}
	Suppose $M$ is a hyper-regular decorated $G$-manifold with an action of linear type, and call $k \coloneqq \dim_{\BC} M - \dim G - \dim_{\BC} V$.
	If the compactification $\overline{M \red G}$ is smooth, then the restriction map
	\begin{equation}
		H^{i} \bigl( \overline{M \red G} \bigr) \longrightarrow H^{i} \bigl( \overline{M \rred G} \bigr)
	\end{equation}
	is an isomorphism for $i < k$ and injective for $i = k$.
\end{theorem}

In this case, the hyper-regularity assumption implies that $k = \dim_{\BC} \bigl(M \rred G\bigr)$.

\begin{proof}
	Call for convenience $X \coloneqq M \red G$.
	By hyper-regularity, $X$ is a smooth manifold and $\check{s}$ a holomorphic function with $0 \in V$ as a regular value.
	The action being of linear type, $X$ is also circle compact, which makes it a decorated manifold, and its almost complex structure is integrable.
	Since $\overline{M}$ is a compact Kähler manifold and $\CO_{\overline{M}}(D_{M})$ is ample, it is a projective variety, and therefore so is $\overline{M} \red G$.
	It is then clear from the identification $\overline{X} \iso \overline{M} \red G$ that $\CO_{\overline{X}} (D_{X})$ is also ample.
	
	We have shown that $X$ is a decorated manifold with $S^{1}$-action of linear type, and since $M \rred G$ is the zero locus of the section on $X$ corresponding to $s$, the result follows from Theorem~\ref{thm:newSommese}.
\end{proof}

\begin{theorem}
	\label{thm:newLHT}
	Suppose $M$ is an $s$-regular decorated Kähler manifold, circle compact for the $S^{1}$-action and with smooth compactification $\overline{M}$.
	Call $\sigma \in \Gamma(\overline{M}, E)$ the section induced by $s$ and $k \coloneqq \dim_{\BC} M - \dim_{\BC} V$.
	Then for every coefficient ring $R$ the restriction map
	\begin{equation}
		H^{i} \bigl( \overline{M}; R \bigr) \longrightarrow H^{i} \bigl( \sigma^{-1} (0); R \bigr)
	\end{equation}
	is an isomorphism for $i < k$ and injective for $i = k$.
\end{theorem}

As in the Lefschetz Hyperplane Theorem, $\overline{M}$ is not required to be projective, but there is a regularity condition on $s$.

\begin{proof}
	We will first of all show that $\sigma$ intersects the $0$ section of $E$ transversely.
	This is clearly the case by $s$-regularity at every point $p \in \sigma^{-1} (0)$ in the interior of $\overline{M}$, as $\sigma$ is identified with $s$ locally around $p$.
	Suppose, instead, that $p \in \sigma^{-1} (0)$ is a boundary point: we will show that $\dif[p]s$ is surjective on $T_{p}D_{M}$, and \emph{a fortiori} on $T_{p} \overline{M}$.
	Using the identification $D_{M} \iso M \reda{c} S^{1}$, choose $\widetilde{p} \in M$ representing $p$.
	If $X$ denotes the vector field generating the $S^{1}$-action on $M$, then $T_{p} D_{M}$ can be realised as the orthogonal complement of $\{ X_{\widetilde{p}}, IX_{\widetilde{p}} \}$ in $T_{\widetilde{p}} M$.
	On the one hand, we know by assumption that $\dif[\widetilde{p}] s$ is surjective on $T_{\widetilde{p}} M$.
	On the other hand
	\begin{equation}
		X_{\widetilde{p}} [s] = \sqrt{-1} ds(\widetilde{p}) = 0 \, ,
	\end{equation}
	where the reader is reminded that $d$ denotes the homogeneity degree of $s$ under the $S^{1}$-action.
	Therefore $IX_{\widetilde{p}} [s]$ also vanishes since $\dif s$ is $\BC$-linear.
	Thus $\dif s$ remains surjective when restricted to the orthogonal complement of $\{ X_{\widetilde{p}}, IX_{\widetilde{p}} \}$, which is enough to conclude that $\sigma$ intersects the $0$ section transversely at $p$.

	Now choose a basis of $V$, to which corresponds a splitting of $E$ as a direct sum of copies of the line bundle $L \coloneqq \CO_{\overline{M}}(D_{M})^{\otimes d}$.
	Since $d > 0$ and the Euler class of $\CO_{\overline{M}} (D_{M})$ is represented by the Kähler form $\omega$, $L$ is positive.
	We will show the result by realising $\sigma^{-1} (0)$ as the intersection of the zero loci of the components $\sigma_{j}$ of $\sigma$.
	To this end, call $Z_{0} = \overline{M}$ and $Z_{j} \coloneqq Z_{j-1} \cap \sigma_{j}^{-1} (0)$ for $j > 0$.
	Since $\sigma \pitchfork 0$ in $E$, one can see that $\sigma_{j} \pitchfork 0$ in $L \rvert_{Z_{j-1}}$, so by induction $Z_{j}$ is a smooth compact Kähler submanifold, on which $L$ remains positive.
	We are then in the setting of the Lefschetz Hyperplane Theorem (see Theorem~\ref{thm:LHT}), and we can conclude that the restriction map
	\begin{equation}
		H^{i} (Z_{j-1}; R) \longrightarrow H^{i} (Z_{j}; R)
	\end{equation}
	is an isomorphism for $i < \dim_{\BC} M - j$ and an embedding for $i = \dim_{\BC} M - j$.
	The result follows by composition.
\end{proof}

\begin{theorem}
	Suppose $M$ is a hyper-regular decorated $G$-manifold with an action of semi-linear type, with $I$ integrable and $\overline{M \red G}$ smooth, and call $k = \dim_{\BC} M - \dim G - \dim_{\BC} V$.
	Then for every coefficient ring $R$ the restriction map
	\begin{equation}
		H^{i} ( \overline{M \red G}; R) \longrightarrow H^{i} ( \overline{M \rred G}; R)
	\end{equation}
	is an isomorphism for $i < k$ and injective for $i = k$.
\end{theorem}

Again, by hyper-regularity $k = \dim M \rred G$.

\begin{proof}
	Under our assumptions, $M \red G$ is a decorated Kähler manifold, regular for the induced function $\check{s}$, circle compact, and with smooth compactification.
	The result follows from Theorem~\ref{thm:newLHT} after noticing that $\overline{M \rred G}$ is the zero locus of the section $\check{\sigma}$ corresponding to $\check{s}$.
\end{proof}

\section{Hypertoric case}
\label{sec:hypertoric}

We now return to the setting that first inspired the present work: that of a hyperkähler manifold $M$ with a tri-hamiltonian action of a compact Lie group $G$.
We will assume the latter is part of a $G \times S^{1}$-action which is Hamiltonian for $\omega_{1}$, with $S^{1}$ acting by hyperkähler rotations.
In that case, the complex moment map $s \coloneqq \mu_{\BC} \colon M \to \fgd_{\BC} \eqqcolon V$ is $I$-holomorphic and $S^{1}$-homogeneous of some degree $d$, which we shall assume positive.
If $0$ is a regular value for the real moment map $\mu_{G} = \mu_{\BR}$, then the same is true for $s$, and $\mu_{G}^{-1} (0) \pitchfork s^{-1} (0)$.

While this makes $M$ (with the Kähler structure $I$) a decorated manifold, the function $s$ is \emph{$G$-equivariant} in general, rather than invariant, unless $G$ is abelian.
For that reason, in order to apply the results from the previous sections we will need to replace $G$ with a torus $T$, and we will say that $M$ is hypertoric.
That being the case, if $M$ is regular (and therefore hyper-regular) then the hyper-reduction $M \rred T$ is the usual hyperkähler quotient.

\begin{theorem}
	Let $M$ be a hyperkähler manifold acted on isometrically by $T \times S^{1}$, with the two factors tri-hamiltonian and $\omega_{1}$-Hamiltonian by positive hyperkähler rotations, respectively.
	Suppose that the resulting decorated $T$-manifold is hyper-regular and that the action is of semi-linear type.
	Then the restriction map
	\begin{equation}
		H^{*} (M \red T; R) \longrightarrow H^{*} (M \rred T; R)
	\end{equation}
	is an isomorphism for every coefficient ring $R$.
	Moreover,
	\begin{equation}
		H^{i} (M \rred T; \BQ) = 0 \qquad \text{for $i > \dim_{\BC} (M \rred T) = \dim_{\BC} M - 2 \dim T$,}
	\end{equation}
	and therefore the same is true for $H^{i} (M \red T; \BQ)$.
	Finally, the hyper-Kirwan map
	\begin{equation}
		\kappa_{T,s} \colon H^{*}_{T} (M; \BQ) \longrightarrow H^{*} (M \rred T; \BQ)
	\end{equation}
	in \emph{rational} cohomology is surjective.
\end{theorem}

The first part of this statement is now a consequence of Theorem~\ref{thm:equivariant_sommese_quotient}.
By Theorem~\ref{thm:KirwanSurj}, the Kirwan map
\begin{equation}
	\kappa \colon H^{*}_{T} (M; \BQ) \longrightarrow H^{*} (M \red T; \BQ)
\end{equation}
is surjective, and the same follows for the hyper-Kirwan map by the isomorphism $H^{*} (M \red T; \BQ) \iso H^{*} (M \rred T; \BQ)$.
The only remaining part is the vanishing statement, which is a consequence of the following.

\begin{theorem}
	Suppose that $M$ is an almost Kähler manifold, circle compact for an $S^{1}$-action preserving the complex structure.
	If $M$ admits an additional complex symplectic form $\omegac$ which is homogeneous of positive degree $d$ for the $S^{1}$-action, then $H^{i} (M)$ vanishes for $i > \dim_{\BC} M$.
\end{theorem}

\begin{proof}
	The argument is essentially the same as \cite[Lemma 5.6]{Nakajima94}, but we repeat it here for completeness.
	By Morse theory with the $S^1$-moment map, the claim will follow if we establish that, for every critical component $\CC$ of $M^{S^1}$, we have the inequality $\lambda_{\CC} + \dim_{\BR} \CC \leq \dim_{\BC} M$, where $\lambda_{\CC}$ denotes the Morse index of $\CC$.

	To show the inequality, suppose $\CC$ is such a component, and for $x \in \CC$ and $k \in \BZ$ let $T_{x} M(k) \subseteq T_{x} M$ denote the subspace of weight $k$ in the isotropy representation.
	Since these subspaces coincide with the eigenspaces of the Hessian of $\mu_{S^{1}}$, one obtains that $\lambda_{\CC} = \sum_{k < 0} \dim_{\BR} T_x M(k)$ and $T_{x} M(0) = T_{x} \CC$, so
	\begin{equation}
		\lambda_{\CC} + \dim_{\BR} \CC = \sum_{k \leq 0} \dim_\BR T_xM(k).
	\end{equation}
	Given that $\omegac$ is homogeneous of degree $d>0$, it pairs the weight spaces $T_{x} M(k)$ and $T_{x} M(d-k)$ non-degenerately.
	Since at least one of $k$ or $d-k$ is positive, we find that every subspace of non-positive weight is matched in dimension by a positive one.
	Therefore, the sum of the dimensions of such subspaces is at most half the dimension of $T_{x} M$, i.e.
	\begin{equation}
		\lambda_{\CC} + \dim_{\BR} \CC \leq \frac{1}{2} \dim_{\BR} M = \dim_\BC M \, ,
	\end{equation}
	which is the desired inequality.
\end{proof}


\end{document}